\documentclass[oneside,english]{amsart}
\usepackage[T1]{fontenc}
\usepackage{geometry}
\usepackage{textcomp}
\usepackage{mathrsfs}
\usepackage{amstext}
\usepackage{amsthm}
\usepackage{amssymb}
\usepackage{stackrel}

\makeatletter

\newcommand{\noun}[1]{\textsc{#1}}
\providecommand{\tabularnewline}{\\}

\numberwithin{equation}{section}
\numberwithin{figure}{section}
\theoremstyle{plain}
\newtheorem{thm}{\protect\theoremname}[section]
\theoremstyle{plain}
\newtheorem{prop}[thm]{\protect\propositionname}
\theoremstyle{plain}
\newtheorem{lem}[thm]{\protect\lemmaname}
\theoremstyle{plain}
\newtheorem{cor}[thm]{\protect\corollaryname}
\theoremstyle{remark}
\newtheorem{rem}[thm]{\protect\remarkname}

\usepackage{bbold}
\usepackage[all]{xy}

\makeatother

\usepackage{babel}
\providecommand{\corollaryname}{Corollary}
\providecommand{\lemmaname}{Lemma}
\providecommand{\propositionname}{Proposition}
\providecommand{\remarkname}{Remark}
\providecommand{\theoremname}{Theorem}

\begin{document}
\title{\noun{Jordan Algebras over icosahedral cut-and-project quasicrystals }}
\author{Daniele Corradetti}
\address{Departamento de Matematica\\
 Universidade do Algarve\\
 Campus de Gambelas\\
 Faro, PT}
\email{d.corradetti@gmail.com}
\author{David Chester}
\address{Quantum Gravity Research\\
 Los Angeles, California \\
 CA 90290, USA\\
}
\email{DavidC@QuantumGravityResearch.org}
\author{Raymond Aschheim}
\address{Quantum Gravity Research\\
 Los Angeles, California \\
 CA 90290, USA\\
}
\email{Raymond@QuantumGravityResearch.org}
\author{Klee Irwin}
\address{Quantum Gravity Research\\
 Los Angeles, California \\
 CA 90290, USA\\
}
\email{Klee@QuantumGravityResearch.org}
\begin{abstract}
In this paper present a general setting for aperiodic Jordan algebras
arising from icosahedral quasicrystals that are obtainable as model
sets of a cut-and-project scheme with a convex acceptance window.
In these hypothesis, we show the existence of an aperiodic Jordan
algebra structure whose generators are in one-to-one correspondence
with elements of the quasicrystal. Moreover, if the acceptance window
enjoys a non-crystallographic symmetry arising from $H_{2},H_{3}$
or $H_{4}$ then the resulting Jordan algebra enjoys the same $H_{2},H_{3}$
or $H_{4}$ symmetry. Finally, we present as special cases some examples
of Jordan algebras over a Fibonacci-chain quasicrystal, a Penrose
tiling, and the Elser-Sloane quasicrystal. 
\end{abstract}

\maketitle
\noindent \textbf{\small{}MSC2020}{\small{}: 52C23, 17B65, 17C50 }{\small\par}

\section{\noun{introduction and motivation}}

Crystallographic Coxeter groups are an indispensable tool in classification
of simple Lie algebras and thus play a prominent role in mathematical
physics. The Chevalley-Serre theorem \cite{Se,Kac} notoriously prevents
the existence of finite-dimensional Lie algebras arising from non-crystallographic
groups such as $H_{2},H_{3}$ and $H_{4}$. Therefore, even though
numerous attempts have been made to encounter algebraic structures
(e.g. \cite{Patera Affine extension,Patera98,Tw99a,Dechant Kac Moody extension})
and integrable systems (e.g. \cite{Shc,FrKo05,Tw99a,Twarock2000})
based non-crystallographic groups, those are far less used and known
in their physical and algebraic applications. Nevertheless, the interest
on non-crystallographic groups has incredibily grown since Setchman
in 1982 found a metallic alloy with icosahedral point group symmetry
\cite{ShBGC}, that was later defined by Levine and Steinhard \cite{LeSt}
as \emph{quasicrystal}, in the sense of a crystal with quasi-periodic
translational symmetry. Indeed, a very general mathematical set-up
for the modelization of such type of quasicrystals was introduced
by Moody and Patera in \cite{Icosians}, making use of an interesting
embedding of the non crystallographic group $H_{4}$ into $E_{8}$
noticed by Elser and Sloane in \cite{Elser Sloane} and then developed
noticed by Shcherbak in \cite{Shc}. In their work, Moody and Patera
showed that a large class of icosahedral quasicrystals are obtainable
as a cut and project scheme on the lattice of icosians or, equivalently,
over an $E_{8}$-lattice. Here, we will have their cut-and-project
scheme as a starting point, focusing on a special class of quasicrystals
arising from an acceptance window that must be a convex set. In this
framework, we show in a constructive way the existence of a Jordan
algebra structure whose generator are in one-to-one correspondence
with elements of the quasicrystal. These Jordan algebras thus fits
into the class of \emph{aperiodic algebras} that were first introduced
by Patera et al. in \cite{Patera98} and later on studied by Twarock\cite{Twarock2000}
and Mazochurck \cite{Ma02}, even though their analysis was limited
to aperiodic Lie algebras and never extended to Jordan algebras. Moreover,
due to the $H_{4}$-symmetry of icosians, if the acceptance window
enjoys an $H_{1}\cong A_{1},H_{2},H_{3}$ or $H_{4}$ symmetry, then
the resulting Jordan algebra enjoys the same $H_{1}\cong A_{1},H_{2},H_{3}$
or $H_{4}$ symmetry. In fact, this class of Jordan algebras are an
algebraic structure naturally endowed with these non-crystallographic
symmetries.

The present work is organized as follows. In section 2 we review basic
definitions in order to fix notations and present crystallographic
groups $H_{2},H_{3}$ and $H_{4}.$ In section 3 we present a reformulation
of the cut-and-project scheme in \cite{Icosians,Elser Sloane} for
both an icosian lattice and an $E_{8}$-lattice. Section 4 is entirely
devoted to a binary operation, called \emph{quasiaddition}, first
introduced in \cite{Berman Quasiaddition} and which will be crucial
in the definition of the aperiodic Jordan algebra. Section 5 deals
with aperiodic algebras, reviewing some known aperiodic algebras from
\cite{Patera98,Twarock2000,Ma02}, then we present the definition
of the aperiodic Jordan algebra $\mathfrak{J}\left(\Xi\right)$ defined
for every quasicrystal $\Xi$ from convex cut-and-project schemes,
and we investigate basic properties of such class of algebras. Finally,
to show that the selected class of quasicrystals is not only broad,
but also highly significant, in section 6 we present some well-known
examples of icosahedral quasicrystals, and thus of aperiodic Jordan
algebras, that fall in this category: the \emph{Fibonacci-chain quasicrystal},
a \emph{Penrose tiling}, a $\mathbb{Z}_{6}$-\emph{quasicrystal} and
the \emph{Elser-Sloane quasicrystal}. 

In the whole paper symbol $\tau$ indicates the \emph{golden mean},
i.e. $\tau=\frac{1}{2}\left(1+\sqrt{5}\right)\approx1.618$ and while,
to avoid any ambiguity, the other root of the equation $x^{2}+x+1$
is referred as $1-\tau$ or $-\tau^{-1}$, i.e. $1-\tau=\frac{1}{2}\left(1-\sqrt{5}\right)\approx-0.618$.

.

\section{\noun{icosians and non-crystallographic root system}s }

The cut-and-project scheme that we will use in later section is a
reformulation of that proposed by Moody and Patera in \cite{Icosians}
and that makes use of icosians developing Elser-Sloane construction
in \cite{Elser Sloane}. The goal of this section is to review all
concepts needed to fully express such mathematical set-up. \emph{Icosians}
are a special set of quaternions that are vertices of a convex regular
4-polytope $\left\{ 3,3,5\right\} $, i.e. a 600-cell, which forms
a group under quaternionic multiplication and, also, a non-crystallographic
root system of type $H_{4}$ with the Euclidean norm. In fact, it
can be easily shown \cite{Dechant21} that the group arising from
icosians is the binary icosahedral group $2I$, the order-120 dicyclic
group, which is double cover of the rotational symmetry group of the
icosahedron. The $H_{4}$-symmetry of the icosian group allows the
Jordan algebras that we will define in sec. 5 to be $H_{4}$-symmetric,
and it is of paramount importance in order to model icosahedral quasicrystals. 

\subsection{\noun{\label{subsec:Euclidean-spaces-and}Euclidean spaces and root
systems}}

Let $\mathbb{E}^{n}$ be an Euclidean space, i.e. a vector space over
the real numbers $\mathbb{R}$ endowed with a positive defined symmetric
bilinear form $\left\langle \cdot\mid\cdot\right\rangle $ called
\emph{inner product}. Then the \emph{lenght} of a vector $v\in\mathbb{E}^{n}$
is given by $\left\Vert v\right\Vert =\sqrt{\left\langle v\mid v\right\rangle },$
and the \emph{angle} $\theta\left(v,w\right)$ of two vectors $v,w\in\mathbb{E}^{n}$
is defined as 
\begin{equation}
\cos\theta\left(v,w\right)=\frac{\left\langle v\mid w\right\rangle }{\left\Vert v\right\Vert \left\Vert w\right\Vert }.
\end{equation}
A \emph{reflection} in the Euclidean space $\mathbb{E}^{n}$ is a
linear transformation $r_{\alpha}$ that fixes all vectors orthogonal
to the vector $\alpha$, i.e.
\begin{equation}
r_{\alpha}\left(v\right)=v-\frac{2\left\langle v\mid\alpha\right\rangle }{\left\langle \alpha\mid\alpha\right\rangle }\alpha,\label{eq:definizione reflection}
\end{equation}
for every $\alpha,v\in\mathbb{E}^{n}$ . Obviously, the composition
of two reflections is still a reflection. Moreover, the inverse of
any reflection is the reflection itself, i.e. $r_{\alpha}^{-1}=r_{\alpha}$,
and, finally, identity itself is a reflection once we set $\alpha=0$.
As a consequence, the reflections of an Euclidean space form a group.
Even more, a straightforward computation shows that $\left\langle r_{\alpha}\left(v\right)\mid r_{\alpha}\left(w\right)\right\rangle =\left\langle v\mid w\right\rangle $
for every $v,w\in\mathbb{E}^{n}$, so that every reflection $r_{\alpha}$
is an isometry. A group generated by reflections (\ref{eq:definizione reflection})
is called a \emph{reflection group.}

Let $\mathbb{E}^{n}$ be an Euclidean space. Then a subset $\Phi$
of vectors in $\mathbb{E}^{n}$ is called a \emph{root system}, if
it contains the opposite of every elements and it is invariant to
every reflection generated by its elements, i.e. 
\begin{equation}
\pm\alpha\in\Phi\,\,\text{ and }\,r_{\alpha}\left(\Phi\right)=\Phi,\label{eq:defRootSystems}
\end{equation}
for every $\alpha\in\Phi$. If the subset $\Phi$ is a root system,
then all vectors $\alpha$ belonging to $\Phi$ are called \emph{roots}
and the dimension $n$ of the embedding space $n$ is called the\emph{
rank} of the system. Finally, for every root system $\Phi$, we call
the \emph{root lattice} of $\Phi$ in $\mathbb{E}^{n}$ as the $\mathbb{Z}$-span
of $\Phi$, i.e. 
\begin{equation}
\Lambda\left(\Phi\right)=\left\{ \beta\in\mathbb{E}^{n}:\beta=n\alpha,\alpha\in\Phi,n\in\mathbb{Z}\right\} .
\end{equation}
It is important to stress out that, by definition (\ref{eq:defRootSystems}),
root systems and reflection groups are linked one another, so that
any root system give rise to the reflection group that leave its roots
invariant and the converse is also true. Finally, all root systems
are notoriously classified in the following families
\begin{itemize}
\item four one-parameter families of type $A_{n},B_{n},C_{n},D_{n}$ of
rank $n\geq1$;
\item a one-parameter family of type $I_{2}\left(m\right)$ of rank $2$
with $m\in\mathbb{N}$. Special cases $I_{2}\left(3\right),I_{2}\left(4\right),I_{2}\left(5\right)$
and $I_{2}\left(6\right)$ are historically labeled as $A_{2},B_{2},H_{2}$
and $G_{2}$ respectively;
\item and six exceptional systems of type $E_{6},E_{7},E_{8},F_{4},H_{3}$
and $H_{4}$.
\end{itemize}
A special class of root systems, called crystallographic systems,
found a major role in Lie theory. Indeed, let us consider the quantities
generated by $\left\langle \beta\mid\alpha\right\rangle $ and $\left\langle \alpha\mid\alpha\right\rangle $
where $\alpha,\beta$ are roots of the same system $\Phi$. If the
ratio between the two roots is always an integer, i.e. 
\begin{equation}
\frac{\left\langle \beta\mid\alpha\right\rangle }{\left\langle \alpha\mid\alpha\right\rangle }\in\mathbb{Z},
\end{equation}
for every $\alpha,\beta\in\Phi$, then the root system is called \emph{crystallographic},
while is called \emph{non crystallographic} otherwise. As famously
proved by Serre \cite{Se}, any crystallographic system allows the
definition of a finite dimensional Lie algebra with Cartan decomposition
related to the root system. On the other hand, non-crystallographic
root system do not satisfy Serre relations and thus do not give rise
to finite dimensional Lie algebra \cite{Kac}. 

\subsection{\noun{\label{subsec:non-crystallographic-root-system}non-crystallographic
root systems}}

A concrete example of a type-$H_{4}$ root system is $\triangle_{4}\subset\mathbb{R}^{4}$
given by the following 120 vertices of a 600-cell $C_{600}$, i.e.

\begin{equation}
\triangle_{4}=\left\{ \begin{array}{cc}
\left(\pm1,0,0,0\right),\frac{1}{2}\left(\pm1,\pm1,\pm1,\pm1\right) & \text{\,\,and all permutation}\\
\frac{1}{2}\left(0,\pm1,\pm\frac{1}{\tau},\pm\tau\right) & \text{\,\,and all even permutation}
\end{array}\right\} .\label{eq:H4definition}
\end{equation}
A subsystem of $\triangle_{4}\subset\mathbb{R}^{4}$ is given by $\triangle_{3}\subset\mathbb{R}^{3}$,
which is a root system of type $H_{3}$, defined by the following
30 vertices of an icosahedron, i.e.

\begin{equation}
\triangle_{3}=\left\{ \begin{array}{cc}
\left(\pm1,0,0\right) & \text{\,\,and all permutation}\\
\frac{1}{2}\left(\pm1,\pm\frac{1}{\tau},\pm\tau\right) & \text{\,\,and all even permutation}
\end{array}\right\} .
\end{equation}
 Finally, $\triangle_{3}\subset\mathbb{R}^{3}$ contains a type-$H_{2}$
root system $\triangle_{2}\subset\mathbb{R}^{2}$ given by the 10
vertices of a decagon, i.e.

\begin{equation}
\triangle_{2}=\left\{ \begin{array}{cc}
\left(\pm1,0\right)\\
\frac{1}{2}\left(\pm1,\pm\tau\right) & \text{\,\,and all permutation}
\end{array}\right\} ,
\end{equation}
so that we have the following chain of root systems $\triangle_{2}\subset\triangle_{3}\subset\triangle_{4}$
that are of type $H_{2},H_{3}$ and $H_{4}$ respectively.

\subsection{\noun{non-crystallographic coxeter groups}}

In order to study reflection groups and root systems in an efficient
way, it is a common practice to define Coxeter groups which are abstract
presentations of reflection groups. Let a \emph{Coxeter matrix} $M$
be a symmetric matrix of finite dimension $n$ with integer off diagonal
entries $\geq2$ and diagonal entries equal to $1$, i.e. $m_{ij}\in\mathbb{N}$,
$m_{ij}=m_{ji},\,\,$$m_{ii}=1$ and $m_{ij}\geq2$ if $i\neq j$.
Then a\emph{ Coxeter group} is an abstract group with presentation
\begin{equation}
\left\langle R_{1},...,R_{n}\mid\left(R_{i}R_{j}\right)^{m_{ij}}=1\right\rangle .
\end{equation}
As Coxeter himself proved in \cite{Coxeter 1934}, every reflection
group is a Coxeter group and, even more, every finite Coxeter group
admits a faithful representation as a finite reflection group of some
Euclidean space. In case of the non-crystallographic root systems
of type $H_{2},H_{3}$ and $H_{4}$, we have a presentation of the
respective Coxeter groups given by $\left\langle R_{1},...,R_{n}\mid\left(R_{i}R_{j}\right)^{m_{ij}}=1\right\rangle $,
with $n$ equal to $2,3$ and $4$ respectively, and with coefficients
$m_{ij}$ of the Coxeter matrices given by the following
\begin{equation}
\left(H_{2}\right)_{ij}=\left(\begin{array}{cc}
1 & 5\\
5 & 1
\end{array}\right),\left(H_{3}\right)_{ij}=\left(\begin{array}{ccc}
1 & 3 & 2\\
3 & 1 & 5\\
2 & 5 & 1
\end{array}\right),\left(H_{4}\right)_{ij}=\left(\begin{array}{cccc}
1 & 3 & 2 & 2\\
3 & 1 & 3 & 2\\
2 & 3 & 1 & 5\\
2 & 2 & 5 & 1
\end{array}\right).
\end{equation}
To give some concrete realisations of such groups, the Coxeter group
$H_{2}$ of order 10 can be realised as the symmetry group of a decagon,
i.e. the reflection group that leaves invariant the roots of the system
$\triangle_{2}$; the group $H_{3}$ of order 120 can be realised
as the full symmetry group of an icosahedron, i.e. the reflection
group that leaves invariant the roots of the system $\triangle_{3}$;
finally, the group $H_{4}$ of order 14400 can be realised as the
symmetry group of a 600-cell, i.e. the reflection group that leaves
invariant the roots of the system $\triangle_{4}$.

\subsection{\noun{the icosian ring\label{subsec:the-icosian-ring}}}

\begin{table}
\begin{centering}
\begin{tabular}{|c|c|c|c|}
\hline 
Root of $\triangle_{4}$ & Icosian in $\mathbb{H}$ & Permutation in $A_{5}$ & Isometry of $\mathbb{R}^{3}$\tabularnewline
\hline 
\hline 
$\pm\left(0,1,0,0\right)$ & $\pm\text{i}$ & $\left(2,3\right)\left(4,5\right)$ & $\frac{1}{2}\left(\begin{array}{ccc}
-1 & \frac{1}{\tau} & -\tau\\
\frac{1}{\tau} & -\tau & -1\\
-\tau & -1 & \frac{1}{\tau}
\end{array}\right)$\tabularnewline
\hline 
$\pm\left(0,0,1,0\right)$ & $\pm\text{j}$ & $\left(2,4\right)\left(5,3\right)$ & $\frac{1}{2}\left(\begin{array}{ccc}
-\tau & 1 & -\frac{1}{\tau}\\
1 & -\frac{1}{\tau} & \tau\\
-\frac{1}{\tau} & \tau & -1
\end{array}\right)$\tabularnewline
\hline 
$\pm\left(0,0,0,1\right)$ & $\pm\text{k}$ & $\left(2,5\right)\left(3,4\right)$ & $\frac{1}{2}\left(\begin{array}{ccc}
-\frac{1}{\tau} & -\tau & 1\\
-\tau & -1 & -\frac{1}{\tau}\\
1 & -\frac{1}{\tau} & -\tau
\end{array}\right)$\tabularnewline
\hline 
$\pm\frac{1}{2}\left(-1,1,1,1\right)$ & $\pm\frac{1}{2}\left(1+\text{i}+\text{j}+\text{k}\right)$ & $\left(3,4,5\right)$ & $\frac{1}{2}\left(\begin{array}{ccc}
-1 & \frac{1}{\tau} & \tau\\
-\frac{1}{\tau} & \tau & -1\\
-\tau & -1 & -\frac{1}{\tau}
\end{array}\right)$\tabularnewline
\hline 
$\pm\frac{1}{2}\left(0,1,-\frac{1}{\tau},\tau\right)$ & $\pm\frac{1}{2}\left(\text{i}-\frac{1}{\tau}\text{j}+\tau\text{k}\right)$ & $\left(1,3\right)\left(4,5\right)$ & $\frac{1}{2}\left(\begin{array}{ccc}
-1 & -\frac{1}{\tau} & \tau\\
-\frac{1}{\tau} & -\tau & -1\\
\tau & -1 & \frac{1}{\tau}
\end{array}\right)$\tabularnewline
\hline 
\end{tabular}
\par\end{centering}
\caption{\label{tab:Correspondence-between-root}Correspondence between root
of the system $\triangle_{4}$ of type $H_{4}$, element of the icosian
group $\mathscr{I}$, elements of the alternating group $A_{5}$ and
the icosahedral symmetries in $\mathbb{R}^{3}$ of which the icosian
group $\mathscr{I}$ is a double cover.}
\end{table}

In section \ref{subsec:non-crystallographic-root-system} we showed
concrete examples of root systems of type of type $H_{2},H_{3}$ and
$H_{4}$. Here, we show a notable quaternionic realisation of the
root system $\triangle_{4}$ of type $H_{4}$, and thus of its subsystems
$\triangle_{3}$ and $\triangle_{2}$ of type $H_{3}$ and $H_{2}$
respectively. With this in mind we identify $\mathbb{R}^{4}$ with
quaternions $\mathbb{H}$ through the canonical identification $\imath$
from $\mathbb{R}^{4}$ to $\mathbb{H}$ given by
\begin{equation}
\imath\left(a,b,c,d\right)\longrightarrow a+b\text{i}+c\text{j}+d\text{k},\label{eq:canonical R4 to H}
\end{equation}
and define the subset of \emph{icosians} $\mathscr{I}\subset\mathbb{H}$
as the quaternions that are image of $\imath\left(\triangle_{4}\right)$,
where elements of $\triangle_{4}\subset\mathbb{R}^{4}$ are explicitly
defined in (\ref{eq:H4definition}). Obviously $\mathscr{I}$ is root
system of type $H_{4}$. Notice that the quaternionic units $1,\text{i},\text{j},\text{k}$
belong to $\mathscr{I}\subset\mathbb{H}$ since they are image of
even permutations of $\left(1,0,0,0\right)$ that belong to $\triangle_{4}$. 

Moreover, it is easy to see that the 120 elements of $\mathscr{I}$
form a group under quaternionic multiplication. This group is known
as the \emph{icosian group} \cite{Conway Sloane} and is isomorphic
to the binary icosahedral group. The reason behind the name it is
due to Hamilton \cite{HamiltonRoot} and lies in its relationship
with the icosahedral rotational group $A_{5}$ of order 60 of which
$\mathscr{I}$ is the double cover. Indeed, the homomorphism from
$\mathscr{I\subset\mathbb{H}}$ to $A_{5}$ given by
\begin{align}
\begin{array}{ccc}
\left(0,1,0,0\right) & \longrightarrow & \left(2,3\right)\left(4,5\right),\\
\left(0,0,1,0\right) & \longrightarrow & \left(2,4\right)\left(5,3\right),\\
\left(0,0,0,1\right) & \longrightarrow & \left(2,5\right)\left(3,4\right),\\
\frac{1}{2}\left(-1,1,1,1\right) & \longrightarrow & \left(3,4,5\right),\\
\frac{1}{2}\left(0,1,-\frac{1}{\tau},\tau\right) & \longrightarrow & \left(1,3\right)\left(4,5\right),
\end{array}
\end{align}
it is well defined and has kernel $\left\{ 1,-1\right\} $. Table
\ref{tab:Correspondence-between-root} shows the correspondence between
roots of $\triangle_{4}$, elements of the icosians group $\mathscr{I}$,
elements of the alternating group $A_{5}$ and, elements of the icosahedral
group as isometries in $\mathbb{R}^{3}$. 

From the icosian group, we define the \emph{icosian ring} $\mathbb{I}$
as the $\mathbb{Z}-$span, i.e. all finite sums of the icosian group
$\mathscr{I}$. Obviously, the icosian ring $\mathbb{I}$ is a subset
of the algebra of quaternions $\mathbb{H}$, and from such algebra
inherits a conjugation and a norm, that we will call \emph{quaternionic
norm} $n$, given by 
\begin{equation}
n\left(x\right)=\left\langle x,x\right\rangle =x\overline{x},
\end{equation}
for every $x\in\mathbb{I}$. It is worth noting that elements of the
icosians group $\mathscr{I}$ are icosians with unitary norm, i.e.
\begin{equation}
\mathscr{I}=\left\{ x\in\mathbb{I}:n\left(x\right)=1\right\} ,
\end{equation}
so that icosians $\mathbb{I}$ of unitary quaternionic norm form a
root system of type $H_{4}$. In a next section we will see that this
is not the only norm we set over the ring of icosians and that ring
of icosians $\mathbb{I}$ endowed with the \emph{Euclidean norm} $n_{\mathbb{E}}$
that we will later define is isomorphic to an $E_{8}$-lattice \cite{Conway Sloane}.

Now, consider the map $*$ on the Dirichlet ring $\mathbb{Z}\left[\tau\right]$
that sends $\tau$ in $1-\tau$, i.e. defined by
\begin{equation}
\left(a+\tau b\right)^{*}=\left(a+b\right)-\tau b,\label{eq:starmap}
\end{equation}
with $a,b\in\mathbb{Z}$. Clearly, the map $*$ can be extended to
the whole icosian ring $\mathbb{I}$ and it will be of paramount importance
in defining icosahedral quasicrystals in the next section.

Finally, an useful way of working with icosians is through their $\mathbb{Z}^{8}$
decomposition, which is the following
\begin{equation}
x=\left(a+\tau b\right)+\left(c+\tau d\right)\text{i}+\left(e+\tau f\right)\text{j}+\left(g+\tau h\right)\text{k},\label{eq:Z8Decomp}
\end{equation}
where $a,b,c,d,e,f,g\in\mathbb{Z}$. Thus, from now on we will use
ambivalently what we will call\emph{ integral coordinates }of a generic
icosian $x\in\mathbb{I}$, i.e. $x=\left(a,b,c,d,e,f,g,h\right)$,
or the quaternionic notation depending on the convenience. Moreover,
it is straightforward to note that 
\begin{equation}
x^{*}=\left(a+b,-b,c+d,-d,e+f,-f,g+h,-h\right),\label{eq:starmap-Z8}
\end{equation}
which gives us an explicit formula for the star map in integral coordinates.

\section{\noun{icosahedral quasicrystals as model sets}}

It is well-known that a general mathematical setup for cut-and-project
quasicrystals is given by \emph{model sets} that arise from cut-and-project
schemes \cite{Mo00}. More specifically, a broad class of icosahedral
quasicrystals, which is the one of our interest, is linked to cut-and-project
schemes over an icosian lattice (e.g. see \cite{Icosians,Mo00}).
In this section we present the notion of \emph{convex cut-and-project
scheme}, introducing a scheme inspired by \cite[sec. 4.1]{Mo00} and
that make use of an icosian lattice. Moreover, we present an equivalent
scheme that make use of an $E_{8}$-lattice. Since for the definition
of the aperiodic Jordan algebra the acceptance window of the scheme
must be convex, we will restrict ourself to this condition. It is
worth noting that the setting of the cut-and-project scheme through
the use of the ``icosian lattice'' or its equivalent $E_{8}$-lattice
is endowed with an $H_{4}$ symmetry that will be useful for Theorem
\ref{thm:Transparency} on the non-crystallographic symmetries of
the Jordan algebra.

\subsection{\noun{model sets and cut-and-project schemes}}

A \emph{cut-and-project scheme} is a lattice $\Lambda$ in an \emph{embedding
space} $\mathbb{R}^{d_{1}}\times G$, where $\mathbb{R}^{d_{1}}$
is a real euclidean space called the \emph{physical space} and, $G$
is a locally compact abelian group that, for our purposes, will simply
be $\mathbb{R}^{d_{2}}$ and that is called the \emph{inner space}.
Along with the previous embedding, a cut-and-project scheme is endowed
with two projections $\pi_{1}$ and $\pi_{2}$ that have image into
$\mathbb{R}^{d_{1}}$ and $\mathbb{R}^{d_{2}}$ respectively. In brief,
a cut-and-project scheme is given by the following collection of spaces
and mappings
\begin{equation}
\begin{array}{ccccc}
\Xi &  & \Lambda &  & \Omega\\
\cap &  & \cap &  & \cap\\
\mathbb{R}^{d_{1}} & \overset{\pi_{1}}{\longleftarrow} & \mathbb{R}^{d_{1}}\times\mathbb{R}^{d_{2}} & \overset{\pi_{2}}{\longrightarrow} & \mathbb{R}^{d_{2}}
\end{array}\label{eq:cp schem-gen}
\end{equation}
such that the restriction $\pi_{1}\mid_{\Lambda}$ is injective in
$\mathbb{R}^{d}$ and the image $\pi_{2}\left(\Lambda\right)$ is
dense in $\mathbb{R}^{d_{2}}$. Given a cut-and-project scheme and
a set $\Omega$ called \emph{acceptance window}, such that the closure
of the interior $\overline{\text{int}\left(\Omega\right)}$ is compact,
then a \emph{model set} $\Xi$ is defined as the points in physical
space that are image of points that once are projected in the inner
space belong to the window $\Omega$, i.e.
\begin{equation}
\Xi=\left\{ \pi_{1}\left(x\right)\in\mathbb{R}^{d_{1}}:x\in\Lambda,\pi_{2}\left(x\right)\in\Omega\right\} .
\end{equation}
In practice, it is often useful to define the model set $\Xi$ through
a \emph{star map }$*$ that goes directly from $\mathbb{R}^{d_{1}}$
to $\mathbb{R}^{d_{2}}$, i.e. 
\begin{equation}
x^{*}=\pi_{2}\left(\pi_{1}\mid_{\Lambda}^{-1}\left(x\right)\right),
\end{equation}
so that the model set is then given by points in the physical space
that are sent by the star map into the acceptance window, i.e. $\Xi=\left\{ x\in\pi_{1}\left(\Lambda\right):x^{*}\in\Omega\right\} $.
In accordance with a common practice, we will call the model set $\Xi$
a \emph{cut-and-project quasicrystal}; moreover, we will add the adjective
\emph{convex} if the window $\Omega$ is a convex subset of $\mathbb{R}^{d_{2}}$.

\subsection{\noun{icosahedral quasicrystals from icosians}}

Consider the cut-and-project scheme obtained embedding a lattice $\mathbb{\widetilde{I}}$
made by two copies of the ring of icosians $\mathbb{I}$, i.e. $\mathbb{\widetilde{I}}\cong\mathbb{I}\times\mathbb{I}$,
in the embedding space $\mathbb{H}\times\mathbb{H}\cong\mathbb{R}^{4}\times\mathbb{R}^{4}$
through the map $x\mapsto\left(x,x^{*}\right)$ where $*$ is the
star map defined in (\ref{eq:starmap}). We then have the following
cut-and-project scheme
\begin{equation}
\begin{array}{ccccc}
\Xi &  & \mathbb{\widetilde{I}} &  & \Omega\\
\cap &  & \cap &  & \cap\\
\mathbb{\mathbb{H}} & \overset{\pi_{\parallel}}{\longleftarrow} & \mathbb{H}\times\mathbb{H} & \overset{\pi_{\perp}}{\longrightarrow} & \mathbb{\mathbb{H}}
\end{array}\label{eq:cut-and-project-Icosians}
\end{equation}
where $\pi_{\parallel}$ and $\pi_{\perp}$ are the trivial projections
and $\Omega$ is the acceptance window which, in this case, is a subset
whose closure is compact. The model set $\Xi$ is then given by
\begin{equation}
\Xi=\left\{ x\in\mathbb{I}:x^{*}\in P\right\} .
\end{equation}
This cut-and-project scheme essentially encompasses the generic situation
for icosahedral symmetry in model sets. In fact, almost all known
icosahedral quasicrystals, obtainable from cut-and-project methods,
arise from this mathematical set up \cite[sec. 4.1]{Mo00}. Nevertheless,
in order to build an aperiodic Jordan algebra over the model set,
we will restrict ourself to the special cases where $\Omega$ is a
convex set. 

\subsection{\noun{icosahedral quasicrystals from an }$E_{8}$\noun{ lattice}}

The ring of icosians enjoys a deep relation with the $E_{8}$ lattice,
which can be made explicit with the aid of a different norm over icosians
(see \cite{Conway Sloane,Icosians}). Indeed, in a previous section
we have seen that icosian ring $\mathbb{I}$ inherits, from its natural
embedding in $\mathbb{H}$, a \emph{quaternionic norm} $n$. We now
notice that the norm of a generic element $x\in\mathbb{I}$, i.e.
$n\left(x\right)$, is a real number of the form $p+q\sqrt{5}$ with
$p,q$ being rational numbers. Thus, we can define a new \emph{Euclidean
norm} $n_{\mathbb{E}}$, given by
\begin{equation}
n_{\mathbb{E}}\left(x\right)=p+q.
\end{equation}
To be more explicit, in integral coordinates the quaternionic norm
$n\left(x\right)$ of an icosian $x=\left(a_{0},b_{0},a_{1},b_{1},a_{2},b_{2},a_{3},b_{3}\right)$
is 
\begin{align}
n\left(x\right)=\stackrel[k=1]{3}{\sum}\left(a_{k}+\sqrt{5}b_{k}\right)^{2}\in\mathbb{Z}\left[\tau\right],
\end{align}
 while the Euclidean norm $n_{\mathbb{E}}\left(x\right)$ is 
\begin{equation}
n_{\mathbb{E}}\left(x\right)=\stackrel[k=1]{3}{\sum}a_{k}^{2}+5b_{k}^{2}+2a_{k}b_{k}\in\mathbb{Z}.
\end{equation}
It can be shown that the icosian ring, with the Euclidean norm, is
isomorphic to an $E_{8}$ lattice $\Lambda\left(E_{8}\right)$ (see
\cite[Ch. 8]{Conway Sloane}). As a consequence of this isomorphism,
icosahedral quasicrystals, that are obtainable as model set of a cut-and-project
scheme over icosians, can also be modeled through a cut-and-project
scheme over an $E_{8}$ lattice. 

Indeed, following \cite{Pa95} let $a_{1},a_{2},a_{3},a_{4}\in\mathbb{R}^{4}$
be 

\begin{equation}
\begin{array}{cccc}
a_{1}= & \frac{1}{2}\left(\frac{1}{\tau},-\tau,0,-1\right), & a_{2}= & \frac{1}{2}\left(0,\frac{1}{\tau},-\tau,1\right),\\
a_{3}= & \frac{1}{2}\left(0,1,\frac{1}{\tau},-\tau\right), & a_{4}= & \frac{1}{2}\left(0,-1,\frac{1}{\tau},\tau\right).
\end{array}\label{eq:a1-a2-a3-a4}
\end{equation}
A cut-and-project scheme equivalent to (\ref{eq:cut-and-project-Icosians})
is obtained embedding the $E_{8}$-lattice $\Lambda\left(E_{8}\right)$
in $\mathbb{R}^{8}\cong\mathbb{R}^{4}\times\mathbb{R}^{4}$, with
the following maps, i.e.
\begin{equation}
\begin{array}{ccccc}
\Xi &  & \Lambda\left(E_{8}\right) &  & \Omega\\
\cap &  & \cap &  & \cap\\
\mathbb{R}^{4} & \overset{\pi_{\parallel}}{\longleftarrow} & \mathbb{R}^{4}\times\mathbb{R}^{4} & \overset{\pi_{\perp}}{\longrightarrow} & \mathbb{R}^{4}
\end{array}\label{eq:E8 c-a-p scheme}
\end{equation}
where $\Omega$ is the acceptance window and the projections $\pi_{\parallel}$
and $\pi_{\perp}$ are defined as follows. Given an element $X\in\Lambda\left(E_{8}\right)$,
i.e.
\begin{equation}
X=c_{1}\alpha_{1}+...+c_{8}\alpha_{8},
\end{equation}
with $c_{1},...,c_{8}\in\mathbb{Z}$ and $\alpha_{1},...,\alpha_{8}$
simple roots of $E_{8}$, we define
\begin{align}
\begin{array}{cc}
\pi_{\parallel}\left(X\right) & =\left(c_{1}+\tau c_{7}\right)a_{1}+\left(c_{2}+\tau c_{6}\right)a_{2}+\left(c_{3}+\tau c_{5}\right)a_{3}+\left(c_{8}+\tau c_{4}\right)a_{4},\\
\pi_{\perp}\left(X\right) & =\left(c_{1}-\frac{1}{\tau}c_{7}\right)a_{1}^{*}+\left(c_{2}-\frac{1}{\tau}c_{6}\right)a_{2}^{*}+\left(c_{3}-\frac{1}{\tau}c_{5}\right)a_{3}^{*}+\left(c_{8}-\frac{1}{\tau}c_{4}\right)a_{4}^{*},
\end{array}\label{eq:projP(X) c1+c7-1}
\end{align}
where $a_{1},a_{2},a_{3},a_{4}$ are the elements defined in (\ref{eq:a1-a2-a3-a4})
and $a_{1}^{*},a_{2}^{*},a_{3}^{*},a_{4}^{*}$ are their respective
images through the the star map in (\ref{eq:starmap}). Notice that
through the canonical identification between $\mathbb{R}^{4}$ and
$\mathbb{H}$ in (\ref{eq:canonical R4 to H}), we have that

\begin{equation}
\begin{array}{cccc}
\iota\left(a_{1}\right)= & \frac{1}{2}\left(\frac{1}{\tau}-\tau\text{i}-\text{k}\right), & \iota\left(a_{2}\right)= & \frac{1}{2}\left(\frac{1}{\tau}\text{i}-\tau\text{j}+\text{k}\right),\\
\iota\left(a_{3}\right)= & \frac{1}{2}\left(\text{i}+\frac{1}{\tau}\text{j}-\tau\text{k}\right), & \iota\left(a_{4}\right)= & \frac{1}{2}\left(-\text{i}+\frac{1}{\tau}\text{j}+\tau\text{k}\right),
\end{array}
\end{equation}
are all icosians, therefore the image of the $E_{8}$-lattice $\Lambda\left(E_{8}\right)$
is the icosian ring $\mathbb{I}$, thus making the model set $\Xi$
an aperiodic subset of the icosian ring.

\section{\noun{Quasiaddition}}

All convex cut-and-project quasicrystals enjoy a simmetry that it
can be easily expressed by a binary operation called \emph{quasiaddition}
\cite{Berman Quasiaddition,Pa95,CCAI} as contraction of ``quasicrystal
addition'' and that can be defined for the icosian ring $\mathbb{I}$,
i.e.
\begin{equation}
x\vdash y=\tau^{2}x-\tau y,
\end{equation}
 for every $x,y\in\mathbb{\mathbb{I}}$. Such binary operation is
not commutative, nor associative, and does not have a unit. Nevertheless
is power-associative and flexible \cite{CCAI-2}. The major properties
of quasiaddition are summarized in the following 
\begin{prop}
Let $x,y\in\mathbb{\mathbb{\mathbb{I}}}$. Then 
\begin{align}
x & \vdash\left(x\vdash\left(x\vdash\left(\ldots\vdash x\right)\right)\right)=x,\label{eq:QAdd1}\\
x & \vdash\left(y\vdash x\right)=\left(x\vdash y\right)\vdash x,\label{eq:QAdd2}\\
x & \vdash\left(y\vdash y\right)=\left(x\vdash x\right)\vdash y=x\vdash y,\label{eq:QAdd3}\\
x & \vdash\left(x\vdash y\right)=y\vdash x.\label{eq:QAdd4}
\end{align}
\end{prop}

\begin{proof}
The first property is just the plain result of $\tau$ being solution
of $\tau^{2}-\tau-1=0$. Indeed,
\[
x\vdash x=\left(1+\tau\right)x-\tau x=x,
\]
and thus the (\ref{eq:QAdd1}) follows. As corollary, quasiaddition
is \emph{power-associative}. \emph{Flexibility} is expressed by (\ref{eq:QAdd2})
which is a straightforward calculation. Since is flexible and not
associative, then obviously $x\vdash\left(x\vdash y\right)\neq\left(x\vdash x\right)\vdash y$.
But, since $x\vdash\left(x\vdash y\right)=\left(1+\tau\right)y-\tau x,$
then, in a rather interesting way, we have 
\begin{align}
x\vdash\left(x\vdash y\right) & =y\vdash x,\\
\left(x\vdash x\right)\vdash y & =x\vdash y,
\end{align}
from which (\ref{eq:QAdd3}) and (\ref{eq:QAdd4}) follow.
\end{proof}
An useful technical Lemma, whose straightforward proof we omit but
can found in \cite{CCAI-2}, is the following
\begin{lem}
\label{lem:repeatedM+n}Let $x,y\in\mathbb{\mathbb{\mathbb{I}}}$.
Then 
\begin{align}
\left(\left(\left(\left(y\vdash x\right)\vdash x\right)\vdash x\right)\vdash\ldots\right)\vdash x & =\left(1+\tau\right)^{k}\left(y-x\right)+x,\label{eq:repeatedn+m}
\end{align}
where $k$ is the number of times the element $x$ is quasiadded to
$y$. 
\end{lem}

Another interesting property of quasiaddition, which is again the
result of a straightforward calculation, is called the \emph{translational
invariance} \cite{Berman Quasiaddition}, i.e. 
\begin{equation}
\left(x+u\right)\vdash\left(y+u\right)=\left(x\vdash y\right)+u,\label{eq:transl-invariance}
\end{equation}
for every $x,y,u\in\mathbb{\mathbb{\mathbb{I}}}$. This property is
useful when working with model sets resulting from a cut-and-project
scheme. Indeed, if quasiaddition acts over quasicrystals as special
kind of inflation of the quasicrystal. Thus, the translational invariance
in (\ref{eq:transl-invariance}) ensures that the inflation of a translated
quasicrystal matches the translation of the already inflated quasicrystal.

Finally, the most important feature of quasiaddition is expressed
by the following 
\begin{thm}
\label{thm:Closure Quasiaddition}If $x,y$ belong to a convex cut-and-project
icosahedral quasicrystal $\Xi$ then also $x\vdash y$ belongs to
it, i.e. $x,y\in\Xi$ implies
\begin{equation}
x\vdash y\in\Xi.
\end{equation}
\end{thm}

\begin{proof}
Let $x,y\in\Xi$ where $\Xi$ is the model set of a cut-and-project
scheme as in (\ref{eq:cut-and-project-Icosians}) with the acceptance
window $\Omega$ that is a convex set, i.e. $\left(1-t\right)x-ty\in\Omega$
for every $x,y\in\Omega$ and $t\in\left[0,1\right]$. Then 
\begin{align*}
\left(x\vdash y\right)^{*} & =\left(1-\tau\right)^{2}x-\left(1-\tau\right)y\\
 & =\left(1-t\right)x+ty,
\end{align*}
 where $t=1-\tau\approx0.618...$ Since $\Omega$ is a convex, then
$\left(x\vdash y\right)^{*}\in\Omega$, thus $x\vdash y\in\Xi.$
\end{proof}
Now, an useful Lemma is the one that shows how quasiaddition behaves
on the integral coordinates of two icosians as defined in (\ref{eq:Z8Decomp}).
\begin{lem}
\label{lem:QuasiAddInteger}Let $x,y\in\mathbb{\mathbb{\mathbb{I}}}$
with $x=\left(a,b,c,d,e,f,g,h\right)$ and $y=\left(a',b',c',d',e',f',g',h'\right)$.
Then
\begin{equation}
x\vdash y=\left(\begin{array}{c}
\left(a-a'\right)+\left(b-b'\right)+a'\\
\left(a-a'\right)+\left(b-b'\right)+b\\
\left(c-c'\right)+\left(d-d'\right)+c'\\
\left(c-c'\right)+\left(d-d'\right)+d\\
\left(e-e'\right)+\left(f-f'\right)+e'\\
\left(e-e'\right)+\left(f-f'\right)+f\\
\left(g-g'\right)+\left(h-h'\right)+g'\\
\left(g-g'\right)+\left(h-h'\right)+h
\end{array}\right)^{t}.
\end{equation}
\label{lem:Quasiaddition Icosians}
\end{lem}

\begin{proof}
The proof is a straightforward once we notice that for an Dirichlet
number $z=\left(a,b\right)=a+\tau b\in\mathbb{Z}\left[\tau\right]$,
then $\tau z=\tau a+\tau^{2}b$ and $\tau^{2}z=\left(1+\tau\right)\left(a+\tau b\right)$
so that multiplication by $\tau$ and $\tau^{2}$ act on integral
coordinates as 
\begin{equation}
\begin{array}{cc}
\tau:\left(a,b\right)\longrightarrow & \left(b,a+b\right),\\
\tau^{2}:\left(a,b\right)\longrightarrow & \left(a+b,a+2b\right).
\end{array}
\end{equation}
In case of $x,y\in\mathbb{\mathbb{\mathbb{I}}}$, as in the hypothesis
of the Lemma, we have 

\begin{equation}
\begin{array}{cc}
\tau^{2}:x\longrightarrow & \left(a+b,a+2b,c+d,c+2d,e+f,e+2f,g+h,g+2h\right),\\
\tau:y\longrightarrow & \left(b',a'+b',c',c'+d',e',e'+f',g',g'+h'\right),
\end{array}
\end{equation}
from which follows the Lemma.
\end{proof}
An immediate Corollary of the previous Lemma, which will be useful
in the next section is the following
\begin{cor}
\label{cor:x|-y=00003Dx}Let $x,y\in\mathbb{\mathbb{\mathbb{I}}}$.Then
$x\vdash y=x$ if and only if $y=x$.
\end{cor}

\section{\noun{Aperiodic Jordan Algebras from icosahedral quasicrystals}}

In this section we briefly review some important aperiodic algebras
studied by Patera and Twarock in \cite{Patera98,Twarock2000,MT03}.
Then we show how to define an aperiodic Jordan algebra for convex
cut-and-project icosahedral quasicrystals. Even though specific quasicrystals
give rise to different features of the algebra, all Jordan algebras
arising from this process enjoy some common properties that we present
here. Finally, to give concreteness to the definition we analyze a
three aperiodic Jordan algebra: one defined over the palindromic Fibonacci-chain
quasicrystal; one for the Penrose tiling in (\ref{eq:ModelPenrose});
and, finally, one for the Elser-Sloane quasicrystal. 

\subsection{\noun{aperiodic algebras }}

Aperiodic algebras were introduced by Patera, Pelantova and Twarock
in \cite{Patera98} with a specific class of aperiodic algebras named
Quasicrystal Lie algebras whose generators are in one to one correspondence
with a one-dimensional quasicrystal which was a modification of a
Fibonacci-chain quasicrystal. Such quasicrystal Lie algebras were
inspired by the usual Witt algebra, but they did not possess a Virasoro
extension, which was an interesting feature that the authors were
aiming for. Later on Twarock defined an aperiodic Witt algebra which
indeed had a Virasoro extension \cite{Twarock2000} and thus allowing
some interesting physical applications in the study of the breaking
of Virasoro symmetry and in the realisation of a class of integrable
systems \cite{Tw99a,TW00b}. A similar algebra, related to the Penrose
tiling we reviewed in (\ref{eq:ModelPenrose}), was then studied by
Twarock and Mazorchuk in \cite{MT03}. 

Unfortunately, Twarock and Mazochurck's work on Penrose Tiling cannot
be extended to the quaternions and, thus, to icosahedral quasicrystal
arising from a cut-and-project system of the kind in (\ref{eq:cut-and-project-Icosians}).
Indeed we have the following
\begin{thm}
Let $\mathbb{I}$ be the ring of icosians and $\left\{ L_{x}\right\} _{x\in\mathbb{I}}$
be a basis for the infinite-dimensional algebra $\mathfrak{W}$ with
bilinear product defined by for every $x,y\in\mathbb{I}$
\begin{equation}
\left[L_{x},L_{y}\right]=f\left(x-y\right)L_{x+y}.\label{eq:Witt-Icosians-1}
\end{equation}
The only $f$that makes $\mathfrak{W}$ a Lie algebra is the zero
function and in that case $\mathfrak{W}$ is an abelian Lie algebra. 
\end{thm}

\begin{proof}
To proof the theorem we notice that for $\mathfrak{W}$ to be a Lie
algebra the product must be antisymmetric and thus $f\left(-x\right)=-f\left(x\right)$.
On the other hand, to be satisfied the Jacobi identity we need 
\begin{equation}
\left[L_{x},\left[L_{y},L_{z}\right]\right]+\left[L_{y},\left[L_{z},L_{x}\right]\right]+\left[L_{z},\left[L_{x},L_{y}\right]\right]=0,
\end{equation}
and since 
\begin{align}
\left[L_{x},\left[L_{y},L_{z}\right]\right] & =f\left(z-y\right)f\left(y+z-x\right)L_{x+y+z},\\
\left[L_{y},\left[L_{z},L_{x}\right]\right] & =f\left(x-z\right)f\left(z+x-y\right)L_{x+y+z},\\
\left[L_{z},\left[L_{x},L_{y}\right]\right] & =f\left(y-x\right)f\left(x+y+z\right)L_{x+y+z},
\end{align}
this means that $f\left(xy\right)=f\left(x\right)\cdot f\left(y\right)$.
Unfortunately this cannot happen over the quaternions. Indeed let
$\text{i,j}$ be two imaginary unit, since $f$ is real and the reals
are abelians, we have $f\left(\text{i}\right)f\left(\text{j}\right)=f\left(\text{j}\right)f\left(\text{i}\right)$,
but on the other hand $\text{ij}=-\text{j}\text{i}$ we have 
\begin{equation}
f\left(\text{i}\text{j}\right)=f\left(\text{i}\right)f\left(\text{j}\right)=f\left(\text{j}\right)f\left(\text{i}\right)=f\left(\text{j}\text{i}\right)=-f\left(\text{ij}\right),
\end{equation}
and thus $f\equiv0$. Then for any quaternion we have $f\left(q\right)=f\left(\text{ij}\left(-\text{j}\text{i}q\right)\right)=f\left(\text{ij}\right)f\left(-\text{j}\text{i}q\right)=0$. 
\end{proof}
Having reviewed why Twarock and Mazochurck approach cannot be extended
to icosians or quaternions in general, we will now switch from Lie
algebras to Jordan algebras showing how aperiodic Jordan algebras
arise quite naturally in this context.

\subsection{\noun{jordan algebras from quasicrystals }}

The previous theorem prevents the extension of aperiodic Witt algebras
to icosahedral quasicrystals arising from cut-and-project schemes
such a (\ref{eq:cut-and-project-Icosians}). On the other hand, the
definition of an aperiodic Jordan algebra is possible for every model
set arising from the cut-and-project scheme (\ref{eq:cut-and-project-Icosians})
with a convex window $\Omega$. Indeed, let $\Xi$ be the convex cut-and-project
quasicrystal arising from a convex window $\Omega$ and let $\mathfrak{J}\left(\Xi\right)$
be the real vector space spanned by $\left\{ L_{x}\right\} _{x\in\Xi}$
equipped with the bilinear product given by

\noun{
\begin{equation}
L_{x}\circ L_{y}=\frac{1}{2}\left(L_{x\vdash y}+L_{y\vdash x}\right),\label{eq:JordanProdDef}
\end{equation}
}for every $x,y\in\Xi$. Since, by Theorem \ref{thm:Closure Quasiaddition}
we have that $x\vdash y\in\Xi$, then the product is well defined.
Moreover, we have the following
\begin{thm}
The algebra $\mathfrak{J}\left(\Xi\right)$ is a Jordan algebra. 
\end{thm}

\begin{proof}
The product in (\ref{eq:JordanProdDef}) is bilinear by definition
and is clearly commutative, i.e.
\begin{equation}
L_{x}\circ L_{y}=L_{y}\circ L_{x}.\label{eq:JordanComm}
\end{equation}
Therefore, it is sufficient to show that the Jordan identity is valid.
Indeed, since (\ref{eq:QAdd1}) and the commutativity of the product,
we have that 
\begin{align*}
\left(L_{x}\circ L_{y}\right)\circ\left(L_{x}\circ L_{x}\right) & =\frac{1}{2}\left(L_{x}\circ L_{y}\right)\circ\left(L_{x\vdash x}+L_{x\vdash x}\right)\\
 & =\left(L_{x}\circ L_{y}\right)\circ L_{x}\\
 & =L_{x}\circ\left(L_{y}\circ L_{x}\right)\\
 & =L_{x}\circ\left(L_{y}\circ\left(L_{x}\circ L_{x}\right)\right),
\end{align*}
thus fulfilling the Jordan identity.
\end{proof}
It is well-known that Jordan algebras have an abundance of idempotent
elements. This algebra is no exception since (\ref{eq:QAdd1}) implies
that all elements of the basis used for its definition are idempotents,
i.e.
\begin{equation}
\left(L_{x}\circ L_{x}\right)=L_{x},
\end{equation}
 for every $x\in\Xi$. 

Clearly there are specific features of the Jordan algebra $\mathfrak{J}\left(\Xi\right)$
that depends on the choice of the quasycrystal $\Xi$ from which is
originated. Nevertheless there are few basic properties that are common
to all Jordan algebras arising from (\ref{eq:JordanProdDef}). Indeed,
an interesting feature is expressed by the the following
\begin{lem}
Let $x,x',y,y'\in\Xi$. If
\begin{equation}
L_{x}\circ L_{y}=\frac{1}{2}\left(L_{x'}+L_{y'}\right),
\end{equation}
then $x+y=x'+y'$.
\end{lem}

\begin{proof}
Is a straightforward implication from Lemma \ref{lem:QuasiAddInteger},
but it can be seen directly since, by definition, $x'=x\vdash y$
and $y'=y\vdash x$. Thus $x'+y'=\left(1+\tau\right)x-\tau y+\left(1+\tau\right)y-\tau x=x+y$.
\end{proof}
From now on, we will suppose the Jordan algebra $\mathfrak{J}\left(\Xi\right)$
to be infinite dimensional, since the only case when this does not
happen is in the trivial case of the acceptance window $\Omega$ being
a point, for which the resulting algebra can hardly be considered
``aperiodic'' being a one-dimensional algebra. 

Beside being infinite dimensional, the algebras $\mathfrak{J}\left(\Xi\right)$
are not unital. This is a consequence of Corollary \ref{cor:x|-y=00003Dx}
for which $x\vdash y=x$ if and only if $y=x$, as shown by the following
\begin{prop}
The Jordan algebras $\mathfrak{J}\left(\Xi\right)$ are not unital.
\end{prop}

\begin{proof}
The proof is given by contraddiction. Let $I\in\mathfrak{J}\left(\Xi\right)$
be the identity, i.e. $I\circ X=X$ for every $X\in\mathfrak{J}\left(\Xi\right)$
, and let $I=\sum\xi_{x}L_{x},$ with $\xi_{x}\in\mathbb{R}$ and
$x\in\Xi$, be a decomposition of $I$. Then, let us consider the
product $I\circ L_{0}=L_{0}$. For the bilinearity we have
\begin{equation}
I\circ L_{0}=\underset{x\in\Xi}{\sum}\xi_{x}L_{x}\circ L_{0}=L_{0},
\end{equation}
but, for Corollary \ref{cor:x|-y=00003Dx} we have that $L_{x}\circ L_{0}=L_{0}$
if and only if $x=0$, therefore the $L_{0}$ component $\xi_{0}$
is equal to $1$ and $I$ is of the form $I=L_{0}+\underset{x\neq0}{\sum}\xi_{x}L_{x}$.
Repeating the same argument for all $x\in\Xi$ we then have that all
$\xi_{x}=1$ and $I$ must be of the form 
\begin{equation}
I=\underset{x\in\Xi}{\sum}L_{x}.
\end{equation}
But, let us consider again the product $I\circ L_{0}$ that we now
know to be of the form
\begin{equation}
I\circ L_{0}=L_{0}\circ L_{0}+\underset{x\neq0}{\sum}L_{x}\circ L_{0}=L_{0},
\end{equation}
this means that the remaining $\sum L_{x}\circ L_{0}$ must be equal
to $0$. Nevertheless, this is impossible since the terms $L_{x}\circ L_{0}$
give rise only to terms of the form $\frac{1}{2}L_{x\vdash0}$ and
$\frac{1}{2}L_{0\vdash x}$ thus with positive coefficients for every
$x\in\Xi$. 
\end{proof}
\begin{rem}
Even though the algebra $\mathfrak{J}\left(\Xi\right)$ is not unital,
it can be extended to an unital algebra with a one-dimensional extension.
Indeed, let $\mathfrak{J}\left(\Xi\right)_{\text{ext}}$ be the vector
space $\mathfrak{J}\left(\Xi\right)\oplus\mathbb{K}$ endowed with
the bilinear product defined by
\begin{equation}
\left(L_{x}+\alpha\right)*\left(L_{y}+\beta\right)=L_{x}\circ L_{y}+\alpha v+\beta u+\alpha\beta,
\end{equation}
where $L_{x},L_{y}\in\mathfrak{J}\left(\Xi\right)$ and $\alpha,\beta\in\mathbb{K}$.
It is easy to show that this is a Jordan algebra with $1\in\mathbb{K}$
as unit element and, that $\mathfrak{J}\left(\Xi\right)_{\text{ext}}$
contains a copy of the original algebra $\mathfrak{J}\left(\Xi\right)$.
Nevertheless, we will not pursue the study of such algebra and we
hold to the the original aperiodic Jordan algebra $\mathfrak{J}\left(\Xi\right)$
. 
\end{rem}

An interesting feature of these algebras is the lack of finite dimensional
sub-algebras that are not the one-dimensional ones generated by an
element $L_{x}$. Indeed, we have the following
\begin{thm}
The algebras $\mathfrak{J}\left(\Xi\right)$ do not have finite dimensional
subalgebras with dimension greater than 1.
\end{thm}

\begin{proof}
Let $\mathfrak{A}$ be a subalgebra in $\mathfrak{J}\left(\Xi\right)$.
If $\text{dim}\mathfrak{A}=1$ then $\mathfrak{A}$ is just the algebra
generated by $L_{x}$ with $x\in\Xi$. If $\text{dim}\mathfrak{A}>1$
then it exist $x,y\in\Xi$ such that $L_{x}$ and $L_{y}$ belong
to $\mathfrak{A}$. If $L_{x}$ and $L_{y}$ are in $\mathfrak{A}$,
then also all elements of the form 
\begin{equation}
L_{x\vdash y},L_{\left(x\vdash y\right)\vdash y},\ldots,L_{\left(\left(x\vdash y\right)\vdash y\right)\cdots\vdash y},\ldots
\end{equation}
 belong to $\mathfrak{A}$. By Lemma \ref{lem:repeatedM+n} the series
if infinite since 
\begin{equation}
\left(1+\tau\right)^{k_{1}}\left(x-y\right)+y=\left(1+\tau\right)^{k_{2}}\left(x-y\right)+y
\end{equation}
 if and only if $k_{1}=k_{2}$, therefore implying that $\mathfrak{A}$
is infinite dimensional and thus proving the theorem.
\end{proof}
As for the simplicity, or the non-simplicity, of the algebra $\mathfrak{J}\left(\Xi\right)$,
at the moment no model set $\Xi$ led to the construction of of a
simple aperiodic Jordan algebra $\mathfrak{J}\left(\Xi\right)$ (see
\cite{CCAI,CCAI-2}). In fact, this is no surprise, since simple infinite-dimensional
Jordan algebras are all classified by Zel'manov in \cite{Ze}. To
have a simple aperiodic Jordan algebra of the type in (\ref{eq:JordanProdDef}),
would mean to have a quasicrystal $\Xi$ arising from a convex window
$\Omega$ that would give rise to an algebra $\mathfrak{J}\left(\Xi\right)$
of the Clifford type or of the Hermitian type, since those are the
only allowed types from Zel'manov classification theorem. In both
cases that would be astonishing and if possible that would be the
result of a very special \emph{ad hoc} construction which would not
be relevant for our practical purposes. 

As a final note, even though the definition of an aperiodic Jordan
Algebra $\mathfrak{J}\left(\Xi\right)$ is always possible for a convex
cut-and-project icosahedral quasicrystal, the $H_{n}$-symmetry of
the quasicrystal, and thus of the Jordan algebra, is the result of
the joint properties of the lattice and the acceptance window $\Omega$.
In fact we have the following
\begin{thm}
\label{thm:Transparency}Let $\mathfrak{J}\left(\Xi\right)$ be an
aperiodic Jordan algebra, where the model set $\Xi$ the set-up defined
in (\ref{eq:cut-and-project-Icosians}) with a convex window $\Omega$.
Then $\mathfrak{J}\left(\Xi\right)$ has an $H_{1}\cong A_{1},H_{2},H_{3}$
or $H_{4}$ if and only if the acceptance window $\Omega$ enjoys
an $H_{1}\cong A_{1},H_{2},H_{3}$ or $H_{4}$ respectively.
\end{thm}

\begin{proof}
Consider a quasicrystal resulting from the set-up defined in (\ref{eq:cut-and-project-Icosians})
such that
\begin{equation}
\Xi=\left\{ x\in\mathbb{I}:x^{*}\in\Omega\right\} .
\end{equation}
Now, let $\rho$ be a reflection of the Coxeter group of $H_{n}$
with $n\in\left\{ 1,2,3,4\right\} $. Since the icosians $\mathbb{I}$
enjoy the $H_{4}$ symmetry, then $\rho\left(\mathbb{I}\right)=\mathbb{I}$.
Thus, if we also have that the acceptance window is invariant through
this reflection, i.e. $\rho\left(\Omega\right)=\Omega$, then also
the quasicrystal is invariant, i.e. $\rho\left(\Xi\right)=\Xi$. Then,
from the definition of a reflection in (\ref{eq:definizione reflection})
follows that 
\begin{equation}
\rho\left(x\vdash y\right)=\rho\left(x\right)\vdash\rho\left(y\right),
\end{equation}
for every $x,y\in\Xi$. The Jordan algebra $\mathfrak{J}\left(\Xi\right)$
it is thus invariant under $\rho$ since canonically isomorphic to
the algebra spanned by $\left\{ L_{\rho\left(x\right)}\right\} _{\rho\left(x\right)\in\Xi}$
equipped with the bilinear product given by

\noun{
\begin{equation}
L_{\rho\left(x\right)}\circ L_{\rho\left(y\right)}=\frac{1}{2}\left(L_{\rho\left(x\right)\vdash\rho\left(y\right)}+L_{\rho\left(y\right)\vdash\rho\left(x\right)}\right),\label{eq:JordanProdDef-1}
\end{equation}
}for every $\rho\left(x\right),\rho\left(y\right)\in\Xi$.
\end{proof}

\section{\noun{A few examples}}

In the previous section we showed some common properties of aperiodic
Jordan Algebra $\mathfrak{J}\left(\Xi\right)$. In this section we
want to give a hint of what these algebras are. Thus, making use of
cut-and-project schemes defined in section 3 we define here four aperiodic
Jordan algebras related to four well-known quasicrystal: the palindromic
\emph{Fibonacci-chain} quasicrystal; a \emph{Penrose tiling}; a $\mathbb{Z}^{6}$-quasicrystal;
the \emph{Elser-Sloane} quasicrystal. It is worth noting that the
Fibonacci-chain quasicrystal enjoy a reflection symmetry, i.e. an
$H_{1}\cong A_{1}$ symmetry; the Penrose tiling a fivefold symmetry
which is a subgroup of $H_{2}$; the specific $\mathbb{Z}^{6}$-quasicrystal
enjoys an $H_{3}$-symmetry; while the Elser-Slone quasicrystal, and
thus the Jordan algebra defined over it, enjoys an $H_{4}$-symmetry. 

\subsection{Fibonacci-chain quasicrystal}

Let us start from the cut-and-project scheme in (\ref{eq:cut-and-project-Icosians})
and set the acceptance window $\Omega$ as $\left(0,1\right]$. Then,
since the star map $*$ leaves $1,\text{i},\text{j},\text{k}$ invariant,
then the resulting model set 
\begin{equation}
\Xi=\left\{ x\in\mathbb{I}:x^{*}\in\Omega\right\} ,\label{eq:PalindromFCModSet}
\end{equation}
 is equivalent to 
\begin{equation}
\Xi\cong\mathcal{F}=\left\{ x\in\mathbb{Z}\left[\tau\right]:x^{*}\in\left(0,1\right]\right\} ,
\end{equation}
 which is the definition of the Fibonacci-chain quasicrystal as model
set (e.g. see \cite{Patera98,CCAI}). While the acceptance window
$\Omega=\left(0,1\right]$ corresponds to the most known Fibonacci-chain
quasicrystal, we are interested here on its symmetric version, i.e.
the palindromic Fibonacci-chain quasicrystal $\mathcal{F}_{\text{pal}}$
which can be obtained setting $\Omega=\left[-\frac{1}{2},\frac{1}{2}\right]$.
A sample of element of the palindromic Fibonacci-chain quasicrystal
can be found in Tab. \ref{tab:Elements-of-the}. 
\begin{table}
\centering{}%
\begin{tabular}{|c|c|c|c|c|c|c|c|c|c|c|c|c|}
\hline 
$n$ &  & $\cdots$ & -4 & -3 & -2 & -1 & 0 & 1 & 2 & 3 & 4 & $\cdots$\tabularnewline
\hline 
$\mathcal{F}_{\text{pal}}\left(n\right)$ &  & $\cdots$ & {\small{}$-2-4\tau$} & {\small{}$-2-3\tau$} & {\small{}$-1-2\tau$} & {\small{}$-1-\tau$} & {\small{}0} & {\small{}$1+\tau$} & {\small{}$1+2\tau$} & {\small{}$2+3\tau$} & {\small{}$2+4\tau$} & $\cdots$\tabularnewline
\hline 
\end{tabular}\caption{\label{tab:Elements-of-the}Elements of the Fibonacci-chain Quasicrystal
$\mathcal{F}_{\text{pal}}\left(n\right)$.}
\end{table}
 Since both the icosians $\mathbb{I}$ and the acceptance window $\Omega=\left[-\frac{1}{2},\frac{1}{2}\right]$
enjoy the 180-degree symmetry given by $\rho\left(x\right)=-x$, then also
the quasicrystal enjoys such symmetry, i.e. $\rho\left(\Xi\right)=\Xi$,
and more specifically, the palindromic Fibonacci-chain quasicrystal
is such that

\begin{equation}
\mathcal{F}_{\text{pal}}\left(-n\right)=-\mathcal{F}_{\text{pal}}\left(n\right).
\end{equation}

The Jordan algebra $\mathfrak{J}\left(\mathcal{F}_{\text{pal}}\right)$,
is defined as the real vector space spanned by $\left\{ L_{x}\right\} _{x\in\mathcal{F}_{\text{pal}}}$
equipped with the bilinear product given by

\begin{equation}
L_{x}\circ L_{y}=\frac{1}{2}\left(L_{x\vdash y}+L_{y\vdash x}\right).\label{eq:definizione Jordan}
\end{equation}
An equivalent definition of the Fibonacci-chain Jordan algebra $\mathfrak{J}\left(\mathcal{F}_{\text{pal}}\right)$
can be obtained on a vector space indexed by integral numbers, i.e.
spanned by $\left\{ L_{n}\right\} _{n\in\mathbb{Z}}$, with bilinear
product defined as
\begin{equation}
L_{n}\circ L_{m}=\frac{1}{2}\left(L_{n'-m'+2n-m}+L_{m'-n'+2m-n}\right),\label{eq:Jordan product nm}
\end{equation}
where $n',m'\in\mathbb{Z}$ and $n'=\mathcal{F}_{\text{pal}}\left(n\right)-\tau n$
and $m'=\mathcal{F}_{\text{pal}}\left(m\right)-\tau m$.
\begin{table}
\begin{centering}
{\small{}}%
\begin{tabular}{|c|c|c|c|c|c|c|}
\hline 
{\small{}$L_{n}\circ L_{m}$} &  & {\small{}$L_{-2}$} & {\small{}$L_{-1}$} & {\small{}$L_{0}$} & {\small{}$L_{1}$} & {\small{}$L_{2}$}\tabularnewline
\hline 
\hline 
{\small{}$L_{-4}$} &  & {\small{}$\frac{1}{2}\left(L_{1}+L_{-7}\right)$} & {\small{}$\frac{1}{2}\left(L_{-8}+L_{3}\right)$} & {\small{}$\frac{1}{2}\left(L_{6}+L_{-10}\right)$} & {\small{}$\frac{1}{2}\left(L_{9}+L_{-12}\right)$} & {\small{}$\frac{1}{2}\left(L_{11}+L_{-13}\right)$}\tabularnewline
\hline 
{\small{}$L_{-3}$} &  & {\small{}$\frac{1}{2}\left(L_{0}+L_{-5}\right)$} & {\small{}$\frac{1}{2}\left(L_{-6}+L_{2}\right)$} & {\small{}$\frac{1}{2}\left(L_{5}+L_{-8}\right)$} & {\small{}$\frac{1}{2}\left(L_{8}+L_{-10}\right)$} & {\small{}$\frac{1}{2}\left(L_{10}+L_{-11}\right)$}\tabularnewline
\hline 
{\small{}$L_{-2}$} &  & {\small{}$L_{-2}$} & {\small{}$\frac{1}{2}\left(L_{-3}+L_{0}\right)$} & {\small{}$\frac{1}{2}\left(L_{3}+L_{-5}\right)$} & {\small{}$\frac{1}{2}\left(L_{6}+L_{-7}\right)$} & {\small{}$\frac{1}{2}\left(L_{8}+L_{-10}\right)$}\tabularnewline
\hline 
{\small{}$L_{-1}$} &  & {\small{}$\frac{1}{2}\left(L_{-3}+L_{0}\right)$} & {\small{}$L_{-1}$} & {\small{}$\frac{1}{2}\left(L_{2}+L_{-3}\right)$} & {\small{}$\frac{1}{2}\left(L_{5}+L_{-5}\right)$} & {\small{}$\frac{1}{2}\left(L_{7}+L_{-6}\right)$}\tabularnewline
\hline 
{\small{}$L_{0}$} &  & {\small{}$\frac{1}{2}\left(L_{-5}+L_{3}\right)$} & {\small{}$\frac{1}{2}\left(L_{2}+L_{-3}\right)$} & {\small{}$L_{0}$} & {\small{}$\frac{1}{2}\left(L_{3}+L_{-2}\right)$} & {\small{}$\frac{1}{2}\left(L_{5}+L_{-3}\right)$}\tabularnewline
\hline 
{\small{}$L_{1}$} &  & {\small{}$\frac{1}{2}\left(L_{-7}+L_{6}\right)$} & {\small{}$\frac{1}{2}\left(L_{5}+L_{-5}\right)$} & {\small{}$\frac{1}{2}\left(L_{-2}+L_{3}\right)$} & {\small{}$L_{1}$} & {\small{}$\frac{1}{2}\left(L_{3}+L_{0}\right)$}\tabularnewline
\hline 
{\small{}$L_{2}$} &  & {\small{}$\frac{1}{2}\left(L_{-8}+L_{8}\right)$} & {\small{}$\frac{1}{2}\left(L_{-6}+L_{7}\right)$} & {\small{}$\frac{1}{2}\left(L_{-3}+L_{5}\right)$} & {\small{}$\frac{1}{2}\left(L_{0}+L_{3}\right)$} & {\small{}$L_{2}$}\tabularnewline
\hline 
{\small{}$L_{3}$} &  & {\small{}$\frac{1}{2}\left(L_{-10}+L_{11}\right)$} & {\small{}$\frac{1}{2}\left(L_{-8}+L_{10}\right)$} & {\small{}$\frac{1}{2}\left(L_{-5}+L_{8}\right)$} & {\small{}$\frac{1}{2}\left(L_{-2}+L_{6}\right)$} & {\small{}$\frac{1}{2}\left(L_{0}+L_{5}\right)$}\tabularnewline
\hline 
{\small{}$L_{4}$} &  & {\small{}$\frac{1}{2}\left(L_{-11}+L_{13}\right)$} & {\small{}$\frac{1}{2}\left(L_{-9}+L_{12}\right)$} & {\small{}$\frac{1}{2}\left(L_{-6}+L_{10}\right)$} & {\small{}$\frac{1}{2}\left(L_{-3}+L_{8}\right)$} & {\small{}$\frac{1}{2}\left(L_{-1}+L_{7}\right)$}\tabularnewline
\hline 
\end{tabular}{\small\par}
\par\end{centering}
\centering{}\caption{\label{tab:Mult-table-F10-1}Multiplication table for a sample of
generators of the aperiodic Jordan algebra $\mathfrak{J}\left(\mathcal{F}_{\text{pal}}\right)$.}
\end{table}
 We will call algebra the integral formulation of the Jordan algebra
$\mathfrak{J}\left(\mathcal{F}_{\text{pal}}\right)$. An analysis
of Tab. \ref{tab:Mult-table-F10-1}, along with few computations allows
to see that $\mathfrak{J}\left(\mathcal{F}_{\text{pal}}\right)$ is
not simple since $L_{0}\mathfrak{J}\left(\mathcal{F}_{\text{pal}}\right)=\left\{ L_{0}\circ x:x\in\mathfrak{J}\left(\mathcal{F}_{\text{pal}}\right)\right\} $
is a proper ideal of the algebra. Similar results holds for all the
examples presented and, as already discussed before, we doubt the
existence of a simple Jordan algebra from this set-up.

\subsection{A Penrose tiling}

Penrose tilings are aperiodic tilings of the plane, introduced by
Roger Penrose in \cite{Pen}, making use of four different prototiles:
two different rhombi and two different tiles called kites and darts.
In this context, different cut-and-project schemes were proposed by
Baake \cite{BaHu,Baake Quasicrystal} and recently refined \cite{ShuMa},
to give rise to different Penrose tilings. For our purposes it is
interesting the specific tiling presented in \cite{MT03} that is
a special case of the general setting proposed by Baake. Indeed, the
vertices of such tilings are given by the aperiodic set 
\begin{equation}
\Xi=\left\{ z\in\mathbb{Z}\left[\tau\right]\triangle\subset\mathbb{C}:z^{*}\in\Omega\right\} ,\label{eq:ModelPenrose}
\end{equation}
where the acceptance window $\Omega$ is the convex pentagon with
vertices $\left\{ 1,\xi,\xi^{2},\xi^{3},\xi^{4}\right\} $, with $\xi=\text{exp}\left(2\pi\text{i}/5\right)$,
and $\triangle$ is the root system of type $H_{2}$ given by 
\begin{equation}
\triangle=\left\{ \pm1,\pm\xi^{2},\pm\left(1+\tau\xi^{2}\right),\pm\left(\tau+\xi^{2}\right),\pm\left(\tau+\tau\xi^{2}\right)\right\} .
\end{equation}
The aperiodic set $\Xi$ is a model set of the following cut-and-project
scheme
\begin{equation}
\begin{array}{ccccc}
\Xi &  & \Lambda &  & \Omega\\
\cap &  & \cap &  & \cap\\
\mathbb{C} & \overset{\pi_{1}}{\longleftarrow} & \mathbb{C}\times\mathbb{C} & \overset{\pi_{2}}{\longrightarrow} & \mathbb{C}
\end{array}\label{eq:cps-Penrose}
\end{equation}
where the lattice $\Lambda$ is given by $\left(\mathbb{Z}+\text{i}\mathbb{Z}\right)\times\left(\mathbb{Z}+\text{i}\mathbb{Z}\right)$
and 
\begin{align}
\pi_{1}\left(x,y,z,w\right) & =x+y\tau+\left(z+w\tau\right)\xi^{2},\\
\pi_{2}\left(x,y,z,w\right) & =x+y\left(1-\tau\right)+\left(z+w\left(1-\tau\right)\right)\xi^{4},
\end{align}
 so that the star-map $*$ sends 
\begin{equation}
\begin{cases}
\begin{array}{cc}
\xi^{*}\longrightarrow & \xi^{2},\\
\tau^{*}\longrightarrow & 1-\tau.
\end{array}\end{cases}
\end{equation}
Such cut-and-project scheme allows the labeling of elements of the
quasicrystal $\Xi$ through the use of quadruples $\left(a,b,c,d\right)$
of $\mathbb{Z}^{4}$, i.e. 
\begin{equation}
\left(a,b,c,d\right)\longrightarrow z=a+b\tau+\left(c+d\tau\right)\xi^{2},\label{eq:decomp a,b,c,d}
\end{equation}
where $a,b,c,d\in\mathbb{Z}$ which are also called the \emph{integer
coordinates} of $z$. A sample of the quasicrystal in its integer
coordinates is given in Tab. \ref{tab:A-sample-of}. 
\begin{table}
\centering{}{\footnotesize{}}%
\begin{tabular}{|c|c|c|c|c|c|c|c|}
\hline 
{\footnotesize{}$\Xi$} & {\footnotesize{}$z_{2}\left(n_{2}\right)$} & {\footnotesize{}$-3$} & {\footnotesize{}$-2$} & {\footnotesize{}$-1$} & {\footnotesize{}$0$} & {\footnotesize{}$1$} & {\footnotesize{}$2$}\tabularnewline
\hline 
\hline 
{\footnotesize{}$z_{1}\left(n_{1}\right)$} &  & {\footnotesize{}$\vdots$} & {\footnotesize{}$\vdots$} & {\footnotesize{}$\vdots$} & {\footnotesize{}$\vdots$} & {\footnotesize{}$\vdots$} & {\footnotesize{}$\vdots$}\tabularnewline
\hline 
{\footnotesize{}$-5$} & {\footnotesize{}$\ldots$} & {\footnotesize{}$\left(-2,-2,-2-4\right)$} & {\footnotesize{}$\left(-2,-2,-1,-2\right)$} & {\footnotesize{}$\left(-2,-2,0,1\right)$} & {\footnotesize{}$\left(-2,-2,0,0\right)$} & {\footnotesize{}$\left(-2,-2,1,1\right)$} & {\footnotesize{}$\left(-2,-2,2,3\right)$}\tabularnewline
\hline 
{\footnotesize{}$-4$} & {\footnotesize{}$\ldots$} & {\footnotesize{}$\left(-1,-3,-4,-6\right)$} & {\footnotesize{}$\left(-1,-3,-3,-5\right)$} & {\footnotesize{}$\left(-1,-3,-2,-3\right)$} & {\footnotesize{}$\left(-1,-3,0,0\right)$} & {\footnotesize{}$\left(-1,-3,1,2\right)$} & {\footnotesize{}$\left(-1,-3,2,3\right)$}\tabularnewline
\hline 
{\footnotesize{}$-3$} & {\footnotesize{}$\ldots$} & {\footnotesize{}$\left(-1,-2,-1,-2\right)$} & {\footnotesize{}$\left(-1,-2,-1,-1\right)$} & {\footnotesize{}$\left(-1,-2,0,-1\right)$} & {\footnotesize{}$\left(-1,-2,0,0\right)$} & {\footnotesize{}$\left(-1,-2,0,1\right)$} & {\footnotesize{}$\left(-1,-2,1,1\right)$}\tabularnewline
\hline 
{\footnotesize{}$-2$} & {\footnotesize{}$\ldots$} & {\footnotesize{}$\left(-1,-1,-1,-2\right)$} & {\footnotesize{}$\left(-1,-1,-1,-1\right)$} & {\footnotesize{}$\left(-1,-1,0,-1\right)$} & {\footnotesize{}$\left(-1,-1,0,0\right)$} & {\footnotesize{}$\left(-1,-1,0,1\right)$} & {\footnotesize{}$\left(-1,-1,1,1\right)$}\tabularnewline
\hline 
{\footnotesize{}$-1$} & {\footnotesize{}$\ldots$} & {\footnotesize{}$\left(0,-1,-2,-2\right)$} & {\footnotesize{}$\left(0,-1,-1,-2\right)$} & {\footnotesize{}$\left(0,-1,-1,-1\right)$} & {\footnotesize{}$\left(0,-1,0,0\right)$} & {\footnotesize{}$\left(0,-1,0,1\right)$} & {\footnotesize{}$\left(0,-1,1,1\right)$}\tabularnewline
\hline 
{\footnotesize{}$0$} & {\footnotesize{}$\ldots$} & {\footnotesize{}$\left(0,0,-1,-2\right)$} & {\footnotesize{}$\left(0,0,-1,-1\right)$} & {\footnotesize{}$\left(0,0,0,-1\right)$} & {\footnotesize{}$0$} & {\footnotesize{}$\left(0,0,0,1\right)$} & {\footnotesize{}$\left(0,0,1,0\right)$}\tabularnewline
\hline 
{\footnotesize{}$1$} & {\footnotesize{}$\ldots$} & {\footnotesize{}$\left(0,1,-1,-2\right)$} & {\footnotesize{}$\left(0,1,-1,-1\right)$} & {\footnotesize{}$\left(0,1,0,-1\right)$} & {\footnotesize{}$\left(0,1,0,0\right)$} & {\footnotesize{}$\left(0,1,1,1\right)$} & {\footnotesize{}$\left(0,1,1,2\right)$}\tabularnewline
\hline 
{\footnotesize{}2} & {\footnotesize{}$\ldots$} & {\footnotesize{}$\left(1,1,-1,-2\right)$} & {\footnotesize{}$\left(1,1,-1,-1\right)$} & {\footnotesize{}$\left(1,1,0,-1\right)$} & {\footnotesize{}$\left(1,1,0,0\right)$} & {\footnotesize{}$\left(1,1,0,1\right)$} & {\footnotesize{}$\left(1,1,1,1\right)$}\tabularnewline
\hline 
{\footnotesize{}3} & {\footnotesize{}$\ldots$} & {\footnotesize{}$\left(1,2,-1,-2\right)$} & {\footnotesize{}$\left(1,2,-1,-1\right)$} & {\footnotesize{}$\left(1,2,0,-1\right)$} & {\footnotesize{}$\left(1,2,0,0\right)$} & {\footnotesize{}$\left(1,2,0,1\right)$} & {\footnotesize{}$\left(1,2,1,1\right)$}\tabularnewline
\hline 
{\footnotesize{}4} & {\footnotesize{}$\ldots$} & {\footnotesize{}$\left(1,3,-4,-7\right)$} & {\footnotesize{}$\left(1,3,-2,-4\right)$} & {\footnotesize{}$\left(1,3,-1,-2\right)$} & {\footnotesize{}$\left(1,3,0,-1\right)$} & {\footnotesize{}$\left(1,3,1,1\right)$} & {\footnotesize{}$\left(1,3,3,4\right)$}\tabularnewline
\hline 
{\footnotesize{}$5$} & {\footnotesize{}$\ldots$} & {\footnotesize{}$\left(2,2,-2,-3\right)$} & {\footnotesize{}$\left(2,2,-1,-2\right)$} & {\footnotesize{}$\left(2,2,-1,-1\right)$} & {\footnotesize{}$\left(2,2,0,0\right)$} & {\footnotesize{}$\left(2,2,1,2\right)$} & {\footnotesize{}$\left(2,2,2,3\right)$}\tabularnewline
\hline 
\end{tabular}\caption{\label{tab:A-sample-of}A sample of elements $z$ in the quasicrystal.
Every element $z$ is given by $z\left(n,m\right)=z_{1}\left(n\right)+z_{2}\left(m\right)\xi^{2}$,
with $z_{1}\left(n\right),z_{2}\left(m\right)\in\mathbb{Z}\left[\tau\right]$
. Moreover, we intended $\left(a,b,c,d\right)\in\mathbb{Z}^{4}$ as
a short notation that corresponds to $z=a+b\tau+\left(c+d\tau\right)\xi^{2}$.}
\end{table}

The aperiodic Jordan algebra $\mathfrak{J}\left(\Xi\right)$ is defined
to be the real vector space spanned by $\left\{ L_{z}\right\} _{z\in\Xi}$
equipped with the bilinear product given by

\noun{
\begin{equation}
L_{z}\circ L_{w}=\frac{1}{2}\left(L_{z\vdash w}+L_{w\vdash z}\right),\label{eq:JordanPenrose}
\end{equation}
 }for every $z,w\in\Xi$. An integral formulation of the same algebra,
i.e. an isomorphic algebra with generators $\left\{ L_{\left(a,b,c,d\right)}\right\} _{\left(a,b,c,d\right)\in\mathbb{Z}^{4}}$,
can be obtained from (\ref{lem:Quasiaddition Icosians}) as

\begin{align}
L_{\left(a,b,c,d\right)}\circ L_{\left(a',b',c',d'\right)}= & \frac{1}{2}\left(L_{\left(a'+\left(a-a'\right)+\left(b-b'\right),b+\left(a-a'\right)+\left(b-b'\right),c'+\left(c-c'\right)+\left(d-d'\right),d+\left(c-c'\right)+\left(d-d'\right)\right)}+\right.\\
 & \left.+L_{\left(a+\left(a-a'\right)+\left(b-b'\right),b'+\left(a-a'\right)+\left(b-b'\right),c+\left(c-c'\right)+\left(d-d'\right),d'+\left(c-c'\right)+\left(d-d'\right)\right)}\right).
\end{align}

\subsection{A $\mathbb{Z}^{6}$\emph{-quasicrystal}}

Three-dimensional icosahedral quasicrystals can be obtained by projection
from three different lattices in a six-dimensional embedding space
$\mathbb{R}^{6}$. These lattices are the primitive cubic lattice
$P$, i.e. $\mathbb{Z}^{6}$ lattice, the face-centred cubic lattice
$2F$, i.e. the root lattice $D_{6}$, and the body-centred cubic
lattice $I$ (reciprocal to $2F$), i.e. the weight lattice $D_{6}$,
respectively \cite{PKK97}. Here we consider a cut and project scheme
analogous to that of the $\mathbb{Z}^{6}$\emph{-}quasicrystal obtained
in \cite{KD86}, where the embedding space is $\mathbb{R}^{6}$ and
the lattice is given by the integral lattice $\mathbb{Z}^{6}$, i.e. 

\begin{equation}
\begin{array}{ccccc}
\Xi &  & \mathbb{Z}^{6} &  & \Omega_{3}\\
\cap &  & \cap &  & \cap\\
\mathbb{R}^{3} & \overset{\pi_{1}}{\longleftarrow} & \mathbb{R}^{3}\times\mathbb{R}^{3} & \overset{\pi_{2}}{\longrightarrow} & \mathbb{R}^{3}
\end{array}\label{eq:cps-Z6}
\end{equation}
 where the maps $\pi_{1}$ and $\pi_{2}$ are given by 
\begin{equation}
\pi_{1}=\left(\begin{array}{cccccc}
1 & \tau & 0 & 0 & 0 & 0\\
0 & 0 & 1 & \tau & 0 & 0\\
0 & 0 & 0 & 0 & 1 & \tau
\end{array}\right),\pi_{2}=\left(\begin{array}{cccccc}
1 & \left(1-\tau\right) & 0 & 0 & 0 & 0\\
0 & 0 & 1 & \left(1-\tau\right) & 0 & 0\\
0 & 0 & 0 & 0 & 1 & \left(1-\tau\right)
\end{array}\right),
\end{equation}
 and the acceptance window $\Omega_{3}$ is given by the rhombic triacontahedron,
i.e. the convex hull obtained from the following $32$ vertices

\begin{equation}
\begin{array}{ccccc}
\left(0,\pm1,\pm\tau\right), & \,\,\, & \left(\pm\tau,0,\pm1\right), & \,\,\, & \left(\pm1,\pm\tau,0\right),\\
 &  & \left(\pm1,\pm1,\pm1\right),\\
\left(0,\pm\tau,\pm\tau^{-1}\right), & \,\,\, & \left(\pm\tau^{-1},0,\pm\tau\right), & \,\,\, & \left(\pm\tau,\pm\tau^{-1},0\right).
\end{array}
\end{equation}
We thus have that the $\mathbb{Z}^{6}$-quasicrystal is given by the
model set 
\begin{equation}
\Xi=\left\{ \pi_{1}\left(x\right)\in\mathbb{R}^{3}:x\in\mathbb{Z}^{6},\pi_{2}\left(x\right)\in\Omega_{3}\right\} .
\end{equation}
An equivalent definition of $\Xi$ can be achieved in a straightforward
way through the aid of an icosian cut and project set-up as in (\ref{eq:cut-and-project-Icosians}),
considering the embedding of $\mathbb{R}^{3}$ in the pure immaginary
quaternions $\text{Pu}\left[\mathbb{H}\right]$ given by 
\begin{equation}
\mathbb{R}^{3}\ni\left(x,y,z\right)\longrightarrow x\text{i}+y\text{j}+z\text{k}\in\text{Pu}\left[\mathbb{H}\right].
\end{equation}
We thus obtain the icosian formulation of the model set of $\Xi$
as
\begin{equation}
\Xi=\left\{ x\in\text{Pu}\left[\mathbb{I}\right]:x^{*}\in\Omega_{3}\right\} ,
\end{equation}
where the star-map is the usual one defined in (\ref{eq:starmap}). 

\subsection{The Elser-Sloane quasicrystal }

The Elser-Sloane quasicrystal $\mathscr{C}$ is a famous quasicrystal
with an $H_{4}$ symmetry that was introduced by Elser and Sloane
in \cite{Elser Sloane} obtained as a projection of the eight dimensional
lattice $\Lambda\left(E_{8}\right)$. In order to follow Elser and
Sloane construction in \cite{Elser Sloane}, we use for this quasicrystal
the cut-and-project scheme over the $E_{8}$ lattice as in (\ref{eq:E8 c-a-p scheme}),
which is equivalent to that in (\ref{eq:cut-and-project-Icosians})
over the icosians lattice $\mathbb{\widetilde{I}}$. Therefore, let
$c^{*}=1/\sqrt{\left(4+2\tau\right)}$ and $\kappa=\tau c^{*}$. Moreover,
let the acceptance window $\Omega_{ES}$ be the convex hull of the
following 720 vertices, i.e.
\begin{equation}
\begin{array}{ccccc}
\frac{1}{2}\kappa\left(\pm2,0,0,0\right), & \,\,\, & \frac{1}{2}\kappa\left(\pm1,\pm1,\pm1,\pm1\right), & \,\,\, & \frac{1}{2}\kappa\left(0,\pm1,\pm\tau,\pm\tau^{-1}\right),\\
\frac{1}{3}\kappa\left(\pm\tau^{2},\pm\tau^{-2},\pm1,0\right), & \,\,\, & \frac{1}{3}\kappa\left(\pm\tau^{2},\pm\tau^{-1},\pm\tau^{-1},\pm\tau^{-1}\right), & \,\,\, & \frac{1}{2}\kappa\left(0,\pm1,\pm\tau,\pm\frac{1}{\tau}\right),\\
\frac{1}{3}\kappa\left(\pm\left(2\tau-1\right),\pm\tau^{-1},\pm\tau,0\right), & \,\,\, & \frac{1}{3}\kappa\left(\pm\left(2\tau-1\right),\pm1,\pm1,\pm1\right), & \,\,\, & \frac{1}{3}\kappa\left(\pm\tau,\pm\tau,\pm\tau,\pm\tau^{-2}\right),\\
\frac{1}{3}\kappa\left(\pm2,\pm2,0,0\right), & \,\,\, & \frac{1}{3}\kappa\left(\pm2,\pm1,\pm\tau,\pm\tau^{-1}\right),
\end{array}
\end{equation}
considered with all choices of signs and any even permutation of the
coordinates. Then, the Elser-Sloane quasicrystal $\mathscr{C}$ is
the model set from the cut-and-project scheme in (\ref{eq:E8 c-a-p scheme})
given by
\begin{equation}
\Xi\cong\mathcal{\mathscr{C}}=\left\{ x\in\pi_{\parallel}\left(X\right):X\in\Lambda\left(E_{8}\right),\pi_{\perp}\left(X\right)\in\Omega_{ES}\right\} ,\label{eq:Elser-Sloane QC}
\end{equation}
for a complete treatment reference see \cite[sec. 4]{Elser Sloane}. 

Even though from an historical point of view the Elser-Sloane quasicrystal
was defined as (\ref{eq:Elser-Sloane QC}), an equally valid quaternionic
realisation of the quasicrystal that make use of the cut-and-project
scheme in (\ref{eq:cut-and-project-Icosians}) is 
\begin{equation}
\mathcal{\mathscr{C}}\cong\Xi=\left\{ x\in\mathbb{I}:x^{*}\in\iota\left(\Omega_{ES}\right)\right\} ,
\end{equation}
where the star map $*$ is that in (\ref{eq:starmap}) and $\iota\left(\Omega_{ES}\right)$
is the image of the acceptance window $\Omega_{ES}$ through the canonical
identification (\ref{eq:canonical R4 to H}) between $\mathbb{R}^{4}$
and $\mathbb{H}$. 

\section{\noun{conclusions and further developments}}

In this work we presented a Jordan structure naturally linked to a
class of icosahedral quasicrystal arising from a cut-and-project setup
with convex acceptance window. The resulting Jordan algebras are infinite
dimensional and fits into the class of aperiodic algebras introduced
by Patera et al. in \cite{Patera98}. Aperiodic Jordan algebras were
firstly defined by the authors in \cite{CCAI} for a Fibonacci chain
quasicrystal and later on extended to a Penrose tiling in \cite{CCAI-2},
but never fully studied in a general setting as this one. In this
work we showed some common features of such algebras, such as the
lack of the unit element (even though they can obviously be extended
to a unital Jordan algebra) and the lack of finite dimensional subalgebras
with dimension greater than $1$. We also showed that such algebras
enjoy a non-crystallographic symmetry when the acceptance window does.
As for the physical applications and future works, aperiodic algebras
found applications in integrable systems \cite{TW00b} and in theoretical
frameworks \cite{Tw99a}. On the other hand, it is well known that
Jordan algebras are relevant in exactly solvable models involving
propagation of solitons. Indeed, it is well known a one-to-one correspondence
between finite dimensional Jordan algebra and multifield Korteweg-de
Vries equations\cite{Sv91}. Similar results hold for the modified
KdV equation \cite{Sv93}, for the Sine-Gordon equation and for generalisations
of the non-linear Schrodinger equation \cite{Sv92}. We thus expect
that propagation of certain kind of excitations in quasicrystal structures
might be related with the Jordan algebras presented here.

\section{Acknowledgment}

This work was funded by the Quantum Gravity Research institute. Authors
would like to thank Fang Fang, Marcelo Amaral, Dugan Hammock, and
Richard Clawson for insightful discussions and suggestions.

\end{document}